\documentclass{amsart}

\numberwithin{equation}{section}

\usepackage{hyperref}
\usepackage{amsmath}
\usepackage{amsthm}
\usepackage{amssymb}
\usepackage{enumerate}
\usepackage[alphabetic]{amsrefs}
\usepackage{tikz}
\usetikzlibrary{arrows,decorations.pathreplacing,decorations.markings,shapes.geometric,through,fit,shapes.symbols,positioning}

\newcommand\fR{\ensuremath{\mathfrak R}}

\newcommand\fg{\ensuremath{\mathfrak g}}
\newcommand\fh{\ensuremath{\mathfrak h}}

\newcommand\fl{\ensuremath{\mathfrak l}}

\newcommand\fs{\ensuremath{\mathfrak s}}

\newcommand\cC{\ensuremath{\mathcal C}}

\newcommand\cO{\ensuremath{\mathcal O}}
\newcommand\cP{\ensuremath{\mathcal P}}

\newcommand\cR{\ensuremath{\mathcal R}}
\newcommand\cS{\ensuremath{\mathcal S}}
\newcommand\cT{\ensuremath{\mathcal T}}
\newcommand\cU{\ensuremath{\mathcal U}}

\newcommand\bbC{\ensuremath{\mathbb C}}

\newcommand\bbQ{\ensuremath{\mathbb Q}}

\newcommand\bbZ{\ensuremath{\mathbb Z}}

\theoremstyle{plain}
\newtheorem{thm}{Theorem}[section]
\newtheorem{prop}[thm]{Proposition}
\newtheorem{lem}[thm]{Lemma}

\theoremstyle{definition}
\newtheorem{rem}[thm]{Remark}
\newtheorem{dfn}[thm]{Definition}
\newtheorem{conv}[thm]{Convention}

\DeclareMathOperator{\wt}{wt}
\DeclareMathOperator{\id}{id}
\DeclareMathOperator{\Gr}{Gr}
\DeclareMathOperator{\End}{End}
\DeclareMathOperator{\Hom}{Hom}
\DeclareMathOperator{\ran}{ran}

\newcommand{\ext}{\Lambda}
\renewcommand{\oint}{\cO}
\newcommand{\ointv}{\oint^{\nu{}}}
\newcommand{\ointq}{\oint^q}
\newcommand{\vect}{\mathrm{Vect}}
\newcommand{\ointone}{\oint^1}
\newcommand{\ointqf}{\cO_{\mathrm{fin}}^q}
\newcommand{\ointvf}{\cO_{\mathrm{fin}}^{\nu}}
\newcommand{\fbg}{\mathfrak{B}_{\mathfrak{g}}}
\newcommand{\uq}{U_q}

\newcommand{\uqg}{\uq(\fg)}

\newcommand{\uvz}{U^{\bbZ}_{\nu}(\fg)}
\newcommand{\uqqg}{U_q^{\bbQ}(\fg)}
\newcommand{\uvq}{U_{\nu}^{\bbQ}(\fg)}
\newcommand{\uoneg}{U_1(\fg)}

\newcommand{\qvm}{\bbQ(\nu^{\frac1m})}

\newcommand{\Qq}{\bbQ(q)}
\newcommand{\Qv}{\bbQ(\nu)}

\newcommand{\zv}{\bbZ[\nu,\nu^{-1}]}
\newcommand{\vlc}{V_\lambda^\bbC}
\newcommand{\braid}{\cR}

\renewcommand{\csc}{\cS \cC}
\newcommand{\e}{\mathrm{ev}}
\renewcommand{\o}{\mathrm{odd}}
\renewcommand{\l}{\mathrm{low}}

\title{Remarks on quantum symmetric algebras}
\author{Alexandru Chirvasitu}
\address{Department of Mathematics \\
University of California \\
Berkeley, CA 94720-3840}
\email{chirvasitua@math.berkeley.edu}

\author{Matthew Tucker-Simmons}
\address{Department of Mathematics \\
University of California \\
Berkeley, CA 94720-3840}
\email{mbtucker@math.berkeley.edu}

\thanks{This research was conducted while the first and second authors were  supported in part by National Science Foundation grants DMS-0901554 and DMS-1066368, respectively.}

\begin{document}

\begin{abstract}
  We examine the quantum symmetric and exterior algebras of finite-dimensional $\uqg$-modules first systematically studied in \cite{BerZwi08} and resolve some of the questions raised therein.
  We show that the difference (in the Grothendieck group) between the quantum symmetric and exterior cubes of a finite-dimensional module is the same as it is classically.
  Furthermore, we show that quantum symmetric algebras are commutative in an appropriate sense.
  We make extensive use of the coboundary structure on the module category.
\end{abstract}

\maketitle

\section{Introduction}
\label{sec:intro}

An important problem in the representation theory of a complex semisimple Lie algebra $\fg$ is to understand the decomposition into simple submodules of the symmetric and exterior powers of a finite-dimensional $\fg$-module.
This motivated Berenstein and Zwicknagl \cite{BerZwi08} to introduce quantum analogues $S_q(V)$ and $\ext_q(V)$ of the classical symmetric and exterior algebras for a finite-dimensional $\uqg$-module $V$.
These are quotients of the tensor algebra of $V$ by relations obtained from the braiding of $V$ with itself.
The relations are homogeneous, and hence both $S_q(V)$ and $\ext_q(V)$ are naturally graded.

The module $V$ is defined to be \emph{flat} \cite{BerZwi08}*{Definition 2.27} if the graded components of $S_q(V)$ (or equivalently of $\ext_q(V)$) have the same dimensions as they do classically.
This is not automatic.
For example, Rossi-Doria \cite{Ros99} showed that the four-dimensional simple $U_q(\fs \fl_2)$-module is not flat.
In fact, Zwicknagl has shown that for any semisimple $\fg$ there are only finitely many flat simple $\uqg$-modules \cite{Zwi}*{Theorems 3.12, 4.23}.
The latter result asserts that $V$ is flat if and only if $S(\overline{V})$, the symmetric algebra of the classical limit of $V$, is Poisson with respect to a bracket coming from a classical $r$-matrix for $\fg$.
Heckenberger and Kolb independently proved that these modules are flat \cite{HecKol04}*{Corollary 6.2}, and applied this to study the associated quantum exterior algebras in \cite{HecKol06}*{Propositions 3.6, 3.7}.
The results of \cite{HecKol04,HecKol06} hold for all non-roots of unity $q \in \bbC^\times$.
By contrast, $q$ is a formal variable in \cite{BerZwi08,Zwi}, which means that the arguments therein will only specialize to transcendental deformation parameters.

In general, $S_q(V)$ is a deformation of a Poisson quotient of $S(\overline{V})$ \cite{BerZwi08}*{Theorem 2.21}.
Thus we may hope that understanding quantum symmetric algebras will lead to some insight into deformations of Poisson structures.

The structure of this paper is as follows.

In Section \ref{sec:notation} we set notation and recall various notions about quantized universal enveloping algebras and their module categories.

In Section \ref{sec:embedding} we show that the homogeneous components of $S_q(V)$ and $\ext_q(V)$ can be embedded naturally back into the tensor algebra, which answers a question of Berenstein and Zwicknagl; see discussion after \cite{BerZwi08}*{Problem 2.12}.

Section \ref{sec:coboundary} contains the first main result of the paper.
Since $S_q(V)$ is an analogue of the symmetric algebra $S(V)$, it is natural to ask whether $S_q(V)$ can be obtained in some sense as an ``enveloping quantum commutative algebra'' of the module $V$.
Theorem \ref{thm:enveloping_comm_alg} shows that this is indeed the case, regardless of whether or not $V$ is flat.
In preparation for this, we begin Section 4 with an extended discussion on continuity and limits of $\uqg$-actions as $q \to 1$, using techniques similar to those in \cite{Wen98}*{\S 2.1}.

In Section \ref{sec:groth} we prove a refined version of \cite{BerZwi08}*{Conjecture 2.26}.
For a finite-dimensional vector space $V$, a simple dimension count shows that
\[ \dim S^3 V - \dim \ext^3 V = (\dim V)^2.   \]
Suppose now that $V$ is a finite-dimensional $\uqg$-module.
Replacing $S^3 V$ and $\ext^3 V$ with $S^3_q V$ and $\ext^3_q V$, respectively, the analogous identity holds when $V$ is flat.
Roughly speaking, Conjecture 2.26 states that the same equality holds for any simple $\uqg$-module, regardless of flatness.
This means that the quantum symmetric and exterior cubes exhibit the same amount of collapsing in the non-flat case.
Theorem \ref{thm:symext} implies the conjecture by analyzing the element $S^3_q V - \ext^3_q V$ in the Grothendieck group of finite-dimensional $\uqg$-modules.

\subsection*{Acknowledgement} 

We would like to thank Sebastian Zwicknagl for the interest shown in this work, and for helpful discussions and suggestions.

\section{Notation and preliminaries}
\label{sec:notation}


\subsection{The Lie algebra}
\label{sec:liealg}

Let $\fg$ be a finite-dimensional complex semisimple Lie algebra with fixed Cartan subalgebra $\fh$ and corresponding root system $\Delta \subseteq \fh^*$.
We fix a choice of positive roots $\Delta^+$  which determines simple roots $\{\alpha_1, \dots, \alpha_r \}$.
The Killing form on $\fg$ determines a nondegenerate symmetric bilinear form $(\cdot \mid \cdot)$ on $\fh^*$, normalized so that $(\alpha \mid \alpha) = 2$ for short roots $\alpha$.
If we denote $d_i = \frac{(\alpha_i \mid \alpha_i)}{2}$, then the Cartan matrix is $(a_{ij})$, determined by $a_{ij} = \frac{(\alpha_i \mid \alpha_j)}{d_i}$.
The fundamental weights are $\{\omega_i, \dots \omega_r  \}$, determined by $(\omega_i \mid \alpha_j ) = \delta_{ij} d_j$.
We denote the weight lattice by $\cP = \bbZ \langle \omega_1, \dots, \omega_r \rangle$, and the set of dominant weights by $\cP^+=\{\lambda\in\cP\ :\ (\lambda\mid\alpha_i)\ge 0,\ \forall i\}$ (this is nothing but the set of non-negative linear combinations of $\omega_i$'s).

Let $m_{ij}=2$, $3$, $4$ or $6$ according to whether $a_{ij}a_{ji}=0$, $1$, $2$ or $3$, respectively. 
Recall that the \emph{braid group} $\fbg$ of $\fg$ is defined by generators $s_i$ for $1 \le i \le r$ and relations \[ s_is_js_i\ldots = s_js_is_j\ldots,\ i\ne j \] with $m_{ij}$ symbols on each side of the `$=$' sign; see \cite{KliSch97}*{\S 6.2.1}.


\subsection{The quantized universal enveloping algebra}
\label{sec:uqg}

Let $q > 0$ be a deformation parameter.
We focus on the positive real case because we need to use the compact real form of $\uqg$.

Denote $q_i = q^{d_i}$. 
The quantum number $[N]_q$ associated to a positive integer $N$ is defined to be $\displaystyle \frac{q^N-q^{-N}}{q-q^{-1}}$, and $[N]_q!$ denotes the quantum factorial $[1]_q [2]_q \dots [N]_q$. 
We may sometimes write $[N]_i$ and $[N]_i!$ for $[N]_{q_i}$ and $[N]_{q_i}!$ respectively.  

The quantized universal enveloping algebra $\uqg$ is the complex, associative, unital algebra generated by $E_i$, $F_i$, and $K_i^{\pm 1}$ for $1 \le i \le r$, with relations ensuring that the $K_i$'s are invertible and commute with one another, together with 
\begin{equation}
  K_i E_j K_i^{-1} = q_i^{a_{ij}} E_j, \label{eq:uqg-rel-ek}
\end{equation}
\begin{equation}
  \label{eq:uqg-rel-fk}
  K_i F_j K_i^{-1} = q_i^{- a_{ij}} F_j,   
\end{equation}
\begin{equation}
  \label{eq:uqg-rel-ef}
  E_i F_j - F_j E_i = \delta_{ij} \frac{K_i - K_i^{-1}}{q_i - q_i^{-1}},   
\end{equation}
and the quantum Serre relations, namely
\begin{equation}
 \sum_{N=0}^{1-a_{ij}} (-1)^N E_i^{(1-a_{ij}-N)}E_jE_i^{(N)} = 0 = \sum_{N=0}^{1-a_{ij}} (-1)^N F_i^{(1-a_{ij}-N)}F_jF_i^{(N)}\label{eq:uqg-rel-serre}
\end{equation}
where $\displaystyle E_i^{(N)}=\frac{E_i^N}{[N]_i!}$ is the $N$th divided power of $E_i$, and $F_i^{(N)}$ is defined analogously. 

Moreover, $\uqg$ is a Hopf algebra when endowed with the comultiplication $\Delta$ and counit $\varepsilon$ defined by
\begin{align*}
  \Delta(E_i) & = E_i \otimes 1+K_i\otimes E_i, \\
  \Delta(F_i) & = F_i \otimes K_i^{-1}+1\otimes F_i,\\
  \Delta(K_i) & = K_i \otimes K_i
\end{align*}
and 
\[   \varepsilon(E_i)=\varepsilon(F_i)=0,\quad \varepsilon(K_i)=1  \]
respectively, and this uniquely determines the antipode.

There is an action of the braid group $\fbg$ on $\uqg$, which we denote by $s_i \mapsto \cT_i$.
Let $w_0$ be the longest word in the Weyl group of $\fg$, and fix a minimal decomposition $ w_0 = w_{i_1} \dots w_{i_n} $ of $w_0$ into a product of simple reflections in the Weyl group.
Here $w_i$ is the reflection corresponding to the simple root $\alpha_i$.
The positive roots of $\fg$ are exhausted by the sequence
\[ \beta_k = w_{i_1} \dots w_{i_{k-1}}(\alpha_k),   \]
and then the \emph{quantum root vectors} in $\uqg$ are defined by
\[  E_{\beta_k} = \cT_{i_1} \dots \cT_{i_{k-1}}(E_{i_k}), \quad  F_{\beta_k} = \cT_{i_1} \dots \cT_{i_{k-1}}(F_{i_k}).  \]
We refer to \cite{KliSch97}*{\S 6.2} for further details.
The \emph{divided powers} of the quantum root vectors are defined to be 
\[ E_\beta^{(N)}=\frac{E_\beta^N}{[N]_{q_\beta}!},\quad F_\beta^{(N)}=\frac{F_\beta^N}{[N]_{q_\beta}!}, \]
where $q_\beta = q^{\frac{(\beta\mid\beta)}2}$.  

When the need for uniform notation arises, we will write $\uoneg$ for the usual enveloping algebra of $\fg$, with generators denoted by $E_i$, $F_i$, $H_i$, $ 1 \le i \le r$.  
This is a Hopf algebra with all elements of $\fg$ primitive.

We give $\uqg$ the $*$-structure called the \emph{compact real form}, determined by
\[  E_i^* = K_iF_i   , \quad F_i^* = E_i K_i^{-1}, \quad K_i^* = K_i,  \] 
with corresponding $*$-structure 
\[ E_i^*=F_i, \quad F_i^*=E_i, \quad H_i^*=H_i\] on $\uoneg$.

We will also need a formal version of $\uqg$.
Let $\nu$ be a formal variable, and define $\nu_i = \nu^{d_i}$.
Define $\uvq$ to be the $\Qv$-algebra with generators $E_i,F_i$, and $K_i^{\pm 1}$ with relations obtained by replacing $q$ with $\nu$ in \eqref{eq:uqg-rel-ek}, \eqref{eq:uqg-rel-fk}, \eqref{eq:uqg-rel-ef}, and \eqref{eq:uqg-rel-serre}; quantum numbers and factorials (and hence divided powers) are defined identically for $\nu$ as for $q$.
The comultiplication and counit of $\uvq$ are defined by the same formulas as for $\uqg$.
The braid group $\fbg$ acts on $\uvq$, and the quantum root vectors and their divided powers are defined as in $\uqg$.
Finally, $\uvz$ denotes the $\zv$-subalgebra of $\uvq$ generated by the divided powers $E_i^{(N)},F_i^{(N)}$ and the elements $K_i^{\pm 1}$.


\subsection{Representations of $\uqg$}
\label{sec:uqgreps}

For each $\lambda \in \cP^+$ we denote by $V_\lambda$ the finite-dimensional irreducible Type 1 representation of $\uqg$ with highest weight $\lambda$, and by $v_\lambda$ a fixed highest weight vector.  
This means that
$E_i v_\lambda = 0$ for all $i$, $V_\lambda = \uqg v_\lambda$, and for $1 \le i \le r$ we have $K_i v_\lambda = q^{(\alpha_i \mid \lambda)}v_\lambda$. 
The same notation will be used for simple $\fg$-modules and their highest weight vectors. 
Consult \cite{ChaPre95}*{\S 10.1} for more details.

We let $\ointq$ denote the category of integrable Type 1 representations of $\uqg$, with morphisms given by all module maps. 
More specifically this is the category of modules generated by the $V_\lambda$'s and closed under arbitrary direct sums. 
Similarly, the category of integrable $\fg$-modules will be denoted by $\ointone$. 
The full subcategory of $\ointq$ consisting of finite-dimensional objects will be denoted by $\ointqf$. 
Here, and throughout the rest of this subsection, $q$ is understood to be greater than 0. 
This allows us to cover both the quantum and the classical cases. 

The categories $\ointq$ and $\ointqf$ have monoidal structures given by the usual tensor product of $\bbC$-vector spaces. 
In both categories, the unit object is the base field $\bbC \cong V_0$, which is a $\uqg$-module via the counit.

The preceding discussion on $\ointq$ carries over almost verbatim to $\uvq$-modules: The category $\ointv$ of integrable $\uvq$-modules consists of direct sums of Type 1 finite-dimensional simple modules.
We use the same symbols $V_\lambda$ as before, as it will be clear from the context whether we are discussing $\uqg$ or $\uvq$-modules. 
Both $\ointv$ and its full subcategory $\ointvf$ consisting of finite-dimensional modules are monoidal with respect to the usual tensor product of $\Qv$-vector spaces.  

Up to a positive scalar multiple, there is a unique inner product on $V_\lambda\in\ointq$, conjugate-linear in the first variable and invariant under the action of $\uqg$ in the sense that
\[ (av, w  ) = (v, a^*w) \]
for all $a \in \uqg$ and $v,w \in V_\lambda$.  We fix the scaling so that $(v_\lambda,v_\lambda)=1$.

For an arbitrary representation $V$, fixing a decomposition 
\[   V \cong \oplus_j V_{\lambda_j} \]
and inner products on the individual summands as explained above determines an inner product on $V$ by making those summands mutually orthogonal. 

On tensor products of representations (each with a decomposition into irreducibles) we take the tensor product of the inner products defined above.

\subsection{Braidings}
\label{sec:braidings}

The monoidal category $\ointq$ is braided.  
The braiding can be realized in many ways, but there is a relatively standard choice, given by
\[ \braid_{V,W} =  \tau \circ R : V \otimes W \to W \otimes V,  \]
where $\tau$ is the tensor flip and $R$ is the $R$-matrix, which lives in an appropriate completion of $\uqg \otimes \uqg$.
We do not give the construction of $R$ explicitly here, but see the proof of Lemma \ref{lem:matrix-coeffs-are-nice} for some details.
The important point is that $R$ acts in tensor products of finite-dimensional representations; see \cite{KliSch97}*{\S 8.3.3}, and also \cite{KamTin09} for a nice description of the completion.

The properties that we need here are
\begin{equation*}
  \begin{gathered}
    R \Delta R^{-1} = \Delta^{op}, \\
    R (v \otimes w) = q^{  (\wt (v) \mid \wt (w) ) } v \otimes w,
  \end{gathered}
\end{equation*}
where $\Delta^{op} = \tau \circ \Delta$ is the opposite comultiplication, and $v,w$ are highest weight vectors.  
From these properties it follows that $R^* = R_{21}$, where $(a \otimes b)^* = a^* \otimes b^*$ and the $*$-structure extends to the completion by continuity.  This implies that
\begin{equation*}
  \braid_{V,W}^* = \braid_{W,V} : W \otimes V \to V \otimes W,
\end{equation*}
where $\braid_{V,W}^*$ is the Hilbert space adjoint of $\braid_{V,W}$ with respect to the inner product defined above in \S \ref{sec:uqgreps}.

This all carries over to $\uvq$, and we use the same notation for analogous objects: $R$ for the $R$-matrix, $\braid_{V,W}$ for the resulting braiding, and so on.


\subsection{Coboundary categories}
\label{sec:cbdry_intro}

In this section we recall the definition of a coboundary category \cite{Dri89}*{\S 3} and discuss the cactus group, which is the coboundary-category analogue of the braid group for a braided monoidal category.

A \emph{coboundary category} is a monoidal category $(\cC, \otimes, 1_\cC)$ with a natural isomorphism $\gamma$ from $\otimes$ to $\otimes^{op}$, i.e.\ for each pair of objects $X,Y$ of $\cC$ an isomorphism 
\[  \gamma_{X,Y} : X \otimes Y \to Y \otimes X  \]
satisfying the relations 
\begin{equation}
  \label{eq:coboundary_rels}
  \gamma_{Y,X} \circ \gamma_{X,Y} = \id, \quad     (\gamma_{Y,Z} \otimes \id_X) \circ \gamma_{X,Y \otimes Z} = (\id_Z \otimes \gamma_{X,Y}) \circ \gamma_{X \otimes Y, Z},
\end{equation}
for all objects $X,Y,Z$ of $\cC$.

In \cite{HenKam06} the authors define, for any objects $A_1, \dots, A_n$ in $\cC$ and  $1 \le p \le r < t \le n$ an isomorphism 
\begin{multline}
  \label{eq:sigma_prq}
  \sigma_{p,r,t} = \id \otimes \gamma_{(A_p \dots A_r),(A_{r+1} \dots A_t)} \otimes \id     : \\ 
  A_1 \dots A_{p-1}(A_p \dots A_r)(A_{r+1} \dots A_t)A_{t+1} \dots A_n \\ 
  \to A_1 \dots A_{p-1}(A_{r+1} \dots A_t)(A_p \dots A_r)A_{t+1} \dots A_n,
\end{multline}
where we have omitted the tensor symbols for readability.

For fixed $p<t$, the relations (\ref{eq:coboundary_rels}) imply that the various ways of composing the maps $\sigma_{p',r,t'}$ for $p \le p' \le r < t' \le t$ to get an isomorphism \[A_1 \dots A_{p-1}(A_p \dots A_t)A_{t+1} \dots A_n \cong A_1 \dots A_{p-1}(A_t \dots A_p)A_{t+1} \dots A_n\] all give the same map, which they call $s_{p,t}$. Alternatively, the $s_{p,t}$ can be defined recursively by
\begin{equation}
  s_{p,p+1} = \sigma_{p,p,p+1}, \quad s_{p,t} = \sigma_{p,p,t} \circ s_{p+1,t} \text{ for $t-p>1$.}  \label{eq:s_pq}
\end{equation}

This motivates the following:

\begin{dfn}
  \label{def:cactus_group}
  The \emph{$n$-fruit cactus group} $J_n$ is the abstract group generated by elements $s_{p,t}$ for $1 \le p < t \le n$ with relations
  \begin{enumerate}[(a)]
  \item $s_{p,t}^2 = 1$;
  \item $s_{p,t}s_{k,l} = s_{k,l}s_{p,t}$ if $p<t$ and $k<l$ are disjoint, i.e. if $t < k$ or $l < p$;
  \item $s_{p,t}s_{k,l} = s_{i,j}s_{p,t}$ if $p \le k < l \le t$, where $i,j$ are determined by $k+j = l+i = p+t$.
  \end{enumerate}
\end{dfn}

For any object $V$ of $\cC$, Lemmas 3 and 4 of \cite{HenKam06} state that the isomorphisms $s_{p,t}$ introduced in (\ref{eq:s_pq}) satisfy these relations, so $J_n$ acts on $V^{\otimes n}$.

There is a homomorphism $J_n \to S_n$ given by 
\begin{equation}
  \label{eq:cactus_quotient}
  s_{p,t} \mapsto \hat{s}_{p,t} =
  \begin{pmatrix}
    1 & \dots & p-1 & p & \dots & t & t+1 & \dots & n \\
    1 & \dots & p -1 & t & \dots & p & t+1 & \dots & n
  \end{pmatrix},
\end{equation}
i.e.\ the involutive permutation which reverses the interval from $p$ to $t$.


\subsection{A coboundary structure on $\ointq$}
\label{sec:cactus}

Taking polar decompositions of the braidings we get maps $\sigma_{V,W}$ which give $\ointq$ the structure of a coboundary category, as shown for instance in \cites{BerZwi08,KamTin09}.
More precisely, the $R$-matrix has the polar decomposition
\[ R = \bar{R}(R^*R)^{\frac12} = \bar{R}(R_{21}R)^{\frac12},   \]
and we then define the \emph{coboundary operators} by 
\begin{equation*}
  \sigma_{V,W} = \tau \circ \bar{R}_{V,W} : V \otimes W \to W \otimes V,
\end{equation*}
where $\bar{R}_{V,W}$ is the action of $\bar{R}$ in $V \otimes W$ and $\tau(v \otimes w) = w \otimes v$.
These operators are unitary by construction and satisfy the coboundary relations \eqref{eq:coboundary_rels}.
For $V = W$ we have that $\sigma_{V,V}$ is a self-adjoint unitary operator on $V \otimes V$, and more generally $\sigma_{V,W}^* = \sigma_{W,V}$.


\subsection{Braidings on super-representations}
\label{sec:super-reps}

In order to treat quantum symmetric and exterior algebras on the same footing, we recall the notion of super-representations.
In particular, we want to show how to extend the braiding (and coboundary structure) on $\ointq$ to the corresponding category of super-representations.

We can do this in a more general setting.
Let $\cC$ be a category satisfying the following hypotheses:
\begin{itemize}
\item $\cC$ is preadditive, i.e.\ $\Hom$-sets are abelian groups and composition of morphisms is bilinear; \cite{Mac98}*{\S I.8}.
\item $\cC$ has finite coproducts;
\item $\cC$ is monoidal, and the coproduct distributes over the tensor product;
\item $\cC$ is equipped with a braiding.
\end{itemize}
By preadditivity, finite products coincide with coproducts, so we simply refer to them as direct sums, and denote them by $\oplus$. 

Given such a category $\cC$, we form the category $\csc = \cC \times \cC$ whose objects are pairs $V = (V_0,V_1)$ of objects of $\cC$ and whose morphisms are pairs $f = (f_0,f_1)$ of morphisms of $\cC$. 
In other words, the objects of $\csc$ are nothing but the $\bbZ/2$-graded, or super-objects of $\cC$, and we will refer to them as such, often identifying $(V_0,V_1)\in\csc$ with the object $V_0\oplus V_1\in \cC$, together with the information of its decomposition into an even part $V_0$ and odd part $V_1$.

The tensor product $V\otimes W$ of two objects $V=(V_0,V_1)$ and $W=(W_0,W_1)$ of $\csc$ is defined by
\[  (V \otimes W)_0 =  (V_0\otimes W_0)\oplus (V_1\otimes W_1), \quad (V \otimes W)_1 = (V_0\otimes W_1)\oplus (V_1 \otimes W_0).  \]
The braiding on $\csc$ is determined by its restrictions to the summands of each $V \otimes W$.   
For $i,j\in\{0,1\}$, the braiding on the summand $V_i\otimes W_j$ of $V\otimes W$ defined above is defined to be the braiding inherited from $\cC$, twisted by the sign $(-1)^{ij}$. 
It is straightforward to check that this does indeed make $\csc$ into a braided monoidal category.
We also note that this discussion carries over verbatim to coboundary structures, so any coboundary monoidal category $\cC$ gives rise to a coboundary monoidal category $\csc$. 
As it will be clear from the context whether we are placing ourselves inside $\cC$ or $\csc$, we will typically denote the braidings (resp. coboundary structures) on $\cC$ and $\csc$ by the same symbol.   

Finally, one last piece of notation: every object $V\in\cC$ has both an even and an odd incarnation in $\csc$, namely $(V,0)$ and $(0,V)$, were $0$ is the zero object of $\cC$.
We denote these by $V_\e$ and $V_\o$, respectively.
In this notation, the monoidal unit for $\csc$ is $1_\e$, where $1$ is the monoidal unit for $\cC$.

The relevance of all of this lies in the observation that if one endows the category $\vect$ of vector spaces over some field with its usual symmetric monoidal structure, then for any $V\in\vect$, the exterior algebra $\Lambda(V)$ can be recovered as the universal commutative algebra in $\cS\vect$ generated by the odd copy $V_\o\in\cS\vect$ of $V$. 
So in a sense, exterior algebras are nothing but symmetric algebras in a different monoidal category. 
We will make all of this more precise for the categories $\ointq$ and $\ointv$ below, after we recall the quantum versions of the symmetric and exterior algebra constructions.


\section{Embedding quantum symmetric and exterior algebras into the tensor algebra}
\label{sec:embedding}

In this section we recall the definitions of quantum symmetric and exterior algebras, and we address Question 2.12 of \cite{BerZwi08}.


\subsection{Quantum symmetric and exterior algebras}
\label{sec:qalgs}

We consider a fixed module $V$ in $\ointq$ and denote $\sigma = \sigma_{V,V}$.   
The spaces of \emph{symmetric} and \emph{antisymmetric} vectors in $V^{\otimes n}$ for $n \ge 2$ are defined by
\begin{equation}
  \label{eq:symvectors}
  \begin{gathered}
  S^n_q V = \{ v \in V^{\otimes n} \mid \sigma_i v = v \text{ for } 1 \le i \le n-1  \},\\
  \ext^n_q V = \{ v \in V^{\otimes n} \mid \sigma_i v = -v \text{ for } 1 \le i \le n-1  \},
  \end{gathered}
\end{equation}
respectively, where $\sigma_i$ is $\sigma$ acting in the $i$ and $i+1$ tensor factors of $V^{\otimes n}$, and the identity in all others.  
The \emph{quantum symmetric algebra} and \emph{quantum exterior algebra} are defined as
\[  S_q(V) = T(V)/\langle \ext^2_q V  \rangle, \quad \ext_q(V) = T(V)/ \langle S^2_q V  \rangle,    \]
respectively.  
As the defining ideals are homogeneous, the algebras are graded, and we denote their graded components by $S^n_q(V)$ and $\ext^n_q(V)$, respectively.

\begin{rem}
  We emphasize that $S^n_q V$, with no parentheses, is a submodule of $V^{\otimes n}$, while $S^n_q(V)$ is a quotient of $V^{\otimes n}$.
  The constructions $V \mapsto S^n_q V$ and $V \mapsto S^n_q(V)$ are functorial in $V$, and similarly for $\ext^n_q$.

  We note also that \cite{BerZwi08} uses the notation $S^n_\sigma V$, $\ext^n_\sigma V$, rather than $S^n_q V$, $\ext^n_q V$, etc.  
  We emphasize the dependence on the parameter $q$ because later we will need to consider what happens as $q$ varies.
\end{rem}

We mentioned above that the point of \S \ref{sec:super-reps} is to allow us to treat quantum exterior algebras as quantum symmetric algebras in a different category.
This works as follows: 

The construction of the quantum symmetric algebra makes sense for any object in a pre-additive coboundary category, so in particular it makes sense in $\cS\ointq$.
For $V \in \ointq$, the quantum symmetric algebra of $V_\o$ in $\cS \ointq$ is exactly $\ext_q(V)$, together with its $\bbZ/2$-grading given by the parity of the homogeneous components. 

In conclusion, $\ext_q(V)$ is, we think, more naturally thought of as an object of $\cS\ointq$, with its parity grading. 
Even though the process outlined in the previous paragraph exhibits it as the quantum \emph{symmetric} algebra on $V_\o\in\cS\ointq$, we will nevertheless use the notation $\ext_q(V_\o)$ when we wish to emphasize this super-representation structure.


\subsection{Symmetrization and antisymmetrization}

The following result exhibits the quantum symmetric and exterior algebra of an object $V\in\ointq$ as subobjects of the tensor algebra $T(V)$ in a canonical way. The resulting projections $T(V)\twoheadrightarrow S_q(V)\hookrightarrow T(V)$ and $T(V)\twoheadrightarrow \ext_q(V)\hookrightarrow T(V)$ can be regarded as $q$-analogues of the usual symmetrization and antisymmetrization operators.
We caution that explicit formulas for these operators in terms of the braiding appear to be very complicated except in simple cases, such as when the braiding satisfies a Hecke-type relation.

\begin{prop}
  \label{prop:embeddings}
  For each $V \in \ointq$ and $n \ge 2$, the natural composite morphisms
  \[
  \begin{gathered}
    S^n_q V \hookrightarrow V^{\otimes n} \twoheadrightarrow  S^n_q(V),\\
    \ext^n_q V \hookrightarrow V^{\otimes n} \twoheadrightarrow \ext^n_q(V)
  \end{gathered}
  \]
  are isomorphisms in $\ointq$.
\end{prop}

\begin{proof}
  We carry out the proof in the first case, and the second follows analogously.  Let us denote $J = \langle \ext^2_q V \rangle$, and the degree $n$ component of this ideal by $J^n = J \cap V^{\otimes n}$.  Then we have
\[ J^n = \sum_{i = 1}^{n-1} \ker (\sigma_i + \id),  \]
and so
\begin{equation}
  (J^n)^\perp =  \left(\sum_{i = 1}^{n-1} \ker (\sigma_i + \id)\right)^\perp = \bigcap_{i=1}^{n-1}  \ker (\sigma_i + \id)^\perp. \label{eq:emb1}
\end{equation}
In \S \ref{sec:cactus} we observed that $\sigma$ is self-adjoint, and hence each $\sigma_i$ is self-adjoint as well.  Since $\sigma_i$ is involutive, we have
\[ \ker (\sigma_i + \id)^\perp = \ker(\sigma_i - \id),  \]
so the right-hand side of (\ref{eq:emb1}) is equal to
\[ \bigcap_{i=1}^{n-1}  \ker(\sigma_i - \id) = S^n_q V.  \]
Then we have $V^{\otimes n} = J^n \oplus S^n_q V$, and the conclusion follows.
\end{proof}

\begin{rem}
  \label{rem:coalg_structures}
  Corollary 2.10 of \cite{BerZwi08} asserts that the modules $S_q V$ and $\ext_q V$ are naturally coalgebras in $\ointq$ with the deconcatenation coproduct.  
  On the other hand $S_q(V)$ and $\ext_q(V)$ are naturally algebras in $\ointq$. 
  Proposition \ref{prop:embeddings} above allows us to transfer the coalgebra structures naturally to $S_q(V)$ and $\ext_q(V)$.  
  It would be interesting to study the compatibility of the algebra and coalgebra structures.
\end{rem}


\section{Quantum symmetric algebras are commutative}
\label{sec:coboundary}

Classically, for a vector space $V$, the symmetric algebra $S(V)$ is the enveloping commutative algebra of $V$ in $\vect$.
It is natural to ask whether the same is true of quantum symmetric algebras, for a suitably defined notion of commutativity. 

Recall that an algebra in a monoidal category $\cC$ is an object $A$ with morphisms $m_A : A \otimes A \to A$ and $u_A : 1_\cC \to A$ satisfying the usual associativity and unit axioms. 
If $\cC$ is equipped with a braiding $(\cR_{V,W})$, there is an obvious notion of commutative algebra $A$, namely requiring $m_A \circ \cR_{A,A} = m_A$.
The algebras $S_q(V)$ are usually not commutative in this sense, essentially because the braidings $\cR_{V,V}$ generally do not have 1 as an eigenvalue.
The following notion is more appropriate for our purposes:

\begin{dfn}
  \label{dfn:comm_alg}
  We say that an algebra object $A$ in a coboundary category $\cC$ is \emph{commutative} if $m_A \circ \gamma_{A,A} = m_A$, where $\gamma_{A,A}$ is the coboundary operator.
  For an object $V$ of $\cC$, an \emph{enveloping commutative algebra of $V$} is a commutative algebra $A$ in $\cC$ with a morphism $V \to A$ such that any morphism from $V$ to a commutative algebra $B$ in $\cC$ factors uniquely through a map of algebras $A \to B$.
\end{dfn}
As usual with such universal constructions, if $V \to A$ exists as above then it is unique up to unique isomorphism.
Our goal in this section is to prove the following result:

\begin{thm}
  \label{thm:enveloping_comm_alg}
  Let $q > 0$ be transcendental and let $V \in \ointqf$.
  \begin{enumerate}[(a)]
  \item The composition $V \to T(V) \to S_q(V)$ makes $S_q(V)$ into an enveloping commutative algebra of $V$ in $\ointq$.
  \item The composition $V_\o \to T(V_\o) \to \Lambda_q(V_\o)$ makes $\Lambda_q(V_\o)$ into an enveloping commutative algebra of $V_\o$ in $\cS \ointq$.
  \end{enumerate}
\end{thm}

We expect that the conclusion of this theorem holds for all positive $q \neq 1$, but our proof does not extend to algebraic deformation parameters.
Before going into the proof, we need some preparation.


\subsection{Classical limits of coboundary structures}
\label{sec:classical_lim}

We now consider the coboundary structure on $\oint$ described in \S \ref{sec:cactus}.  
Our goal is to show that, for a fixed object $V$, the action of the cactus group $J_n$ on $V^{\otimes n}$ is continuous with respect to the parameter $q$, and specializes to the usual action of $S_n$ at $q=1$.

Some work is needed in order to make this goal precise.  
For each  $q > 0$ there is a Hopf algebra $\uqg$ with its corresponding category $\ointq$ of integrable modules.
In order to make rigorous arguments about continuity of families of operators, for each $\lambda \in \cP^+$ we need to somehow be able to identify the underlying vector spaces of the simple $\uqg$-modules $V_\lambda$ for all $q$.  
Kashiwara's notion of global crystal basis \cite{Kas91}, \cite{Lus10}*{\S 14.4} is a tool which will allow us to make this identification coherently.  
The crystal basis is constructed in the context of a formal deformation parameter $\nu$.  
We want to work with complex deformation parameters, and therefore we will need to make sure that the arguments involving the crystal basis can be specialized to complex numbers.

\begin{dfn}
  \label{dfn:universal-v-lambda}
  For $\lambda \in \cP^+$, we let $\vlc$ be the complex vector space with basis $B_\lambda$, where $B_\lambda$ is the global crystal basis for the $\uvq$-module $V_\lambda$.
\end{dfn}

We will use $\vlc$ as the underlying vector space of the simple $\uqg$-modules with highest weight $\lambda$ for all $q > 0$.
The relevance of the crystal basis is that the pair $(V_\lambda, B_\lambda)$ is a \emph{based module} for $\uqqg$ in Lusztig's sense \cite{Lus10}*{\S 27.1.2, \S 27.1.4}.
The property that we need is that $\uvz$ preserves the $\zv$-submodule of $V_\lambda$ generated by $B_\lambda$.
This is important for us because Laurent polynomials are specializable at any nonzero complex number.

\begin{conv}
  Let $V$ be a $\Qv$-vector space with a distinguished basis $B$, and let $T : V \to V$ be a linear map.
  If $T$ preserves the $\zv$-span of $B$ (or equivalently, the matrix coefficients of $T$ with respect to $B$ lie in $\zv$) then we will say that $T$ \emph{acts on $B$ by Laurent polynomials}.
\end{conv}

We will now define an action of each $\uqg$ on $\vlc$ for $q > 0$.
We need to treat the $q \neq 1$ and $q = 1$ cases separately because the generators $K_i$ do not have analogues in $U_1(\fg)$.

For $q \neq 1$, we define the action of the generators $E_i,F_i,K_i$ of $\uqg$ on the basis $B_\lambda$ by specializing the actions of the generators $E_i,F_i,K_i$ of $\uvq$, respectively, at $\nu = q$.
This specialization is possible because the matrix coefficients of the generators of $\uvq$ with respect to $B_\lambda$ lie in $\zv$, and hence can be evaluated at $\nu = q$ for $q \neq 0$.
This defines a representation since the relations of $\uqg$ are obtained from those of $\uvq$ by replacing $\nu$ with $q$.

For $q = 1$, we define operators $e_i,f_i,$ and $h_i$ on $\vlc$ by specializing the actions on $B_\lambda$ of the elements $E_i,F_i$, and $\frac{K_i -K_i^{-1}}{\nu_i - \nu_i^{-1}}$ of $\uvq$, respectively, at $\nu = 1$.
Note that $\frac{K_i -K_i^{-1}}{\nu_i - \nu_i^{-1}}$ scales a vector of weight $\mu$ by 
\[ \frac{\nu_i^{(\mu \mid \alpha_i^{\vee} ) }- \nu_i^{-(\mu \mid \alpha_i^\vee)}}{\nu_i - \nu_i^{-1}}, \]
which is a Laurent polynomial, and hence specializing the action of this element at $\nu = 1$ makes sense.
We now show that these operators define a representation of $U_1(\fg)$.

\begin{lem}
  \label{lem:specialize-to-one}
  The operators $e_i,f_i$, and $h_i$ on $\vlc$ satisfy the relations
  \[  [h_i, e_j] = a_{ij} e_j, \qquad [h_i,f_j] = -a_{ij} f_j, \qquad [e_i, f_j] = \delta_{ij} h_i,   \]
  as well as the Serre relations
  \[ \sum_{N=0}^{1-a_{ij}} (-1)^N \binom{1-a_{ij}}{N} e_i^{1-a_{ij}-N}e_je_i^{N} = 0 \]
and
\[ \sum_{N=0}^{1-a_{ij}} (-1)^N \binom{1-a_{ij}}{N} f_i^{1-a_{ij}-N}f_jf_i^{N} = 0 ,\]
  and hence $E_i \mapsto e_i$, $F_i \mapsto f_i$, $H_i \mapsto h_i$ determines a representation of $U_1(\fg)$ on $\vlc$.
\end{lem}

\begin{rem}
  \label{rem:limit-is-what-it-should-be}
  It follows from Lemma \ref{lem:continuity-of-uqg-actions} below that $\vlc$ is the simple $\fg$-module of highest weight $\lambda$.
\end{rem}

\begin{proof}
  The Serre relations for the $e_i$ and $f_i$ follow immediately from the quantum Serre relations \eqref{eq:uqg-rel-serre}, and the relation $[e_i,f_j] = \delta_{ij}h_i$ follows from \eqref{eq:uqg-rel-ef}.

  In $\uvq$ we have the relation
  \begin{align*}
     \left[\frac{K_i - K_i^{-1}}{\nu_i - \nu_i^{-1}}, E_j\right] & = \frac{\nu_i^{a_{ij}}-1}{\nu_i - \nu_i^{-1}} (E_jK_i + K_i^{-1}E_j) \\
     & = \frac{\nu_i^{2a_{ij}}-1}{(\nu_i - \nu_i^{-1})(\nu_i^{a_{ij}}+1)} (E_jK_i + K_i^{-1}E_j).
  \end{align*}
  Since $\nu_i - \nu_i^{-1}$ divides $\nu_i^{2a_{ij}}-1$ in $\zv$, the fraction is specializable at $\nu = 1$, and its value is $\frac{a_{ij}}{2}$.
  Since $K_i$ specializes to the identity, this gives the relation $[h_i,e_j] = a_{ij} e_j$.
  The relation $[h_i,f_j] = -a_{ij} f_j$ follows similarly.
\end{proof}

Hence $\vlc$ carries an action of $\uqg$ for all $q > 0$.
Next we will show that these actions are continuous in $q$ in an appropriate sense.
For a fixed $i$, the action of $E_i \in \uqg$ determines a family of operators on $\vlc$ parametrized by $q  > 0$, and similarly for the $F_i$.
For the Cartan parts of the algebras, for uniformity of notation we define elements $H_i \in \uqg$ and $H_i \in \uvq$ by
\begin{equation*}
  H_i = \frac{K_i - K_i^{-1}}{q_i - q_i^{-1}} \quad \text{ and } \quad H_i = \frac{K_i - K_i^{-1}}{\nu_i - \nu_i^{-1}},
\end{equation*}
respectively, for $1 \le i \le r$.
Then the action of $H_i \in \uqg$ also determines a family of operators on $\vlc$ indexed by $q > 0$.

\begin{lem}
  \label{lem:continuity-of-uqg-actions}
  For each $i$, the family of operators on $\vlc$ given by the action of $E_i \in \uqg$ for $q > 0$ is continuous in $q$, and similarly for the families coming from $F_i$ and $H_i$.
\end{lem}

\begin{proof}
  This follows from the fact that $E_i,F_i,H_i \in \uvq$ act on the basis $B_\lambda$ by Laurent polynomials in $\nu$, and the actions of $E_i,F_i,H_i \in \uqg$ are obtained by specializing to $\nu = q$.
\end{proof}

Now we would like to extend this identification to non-simple modules and to tensor products.
For a finite-dimensional $U_{q_0}(\fg)$-module $V$, we want to construct a complex vector space $V^\bbC$ carrying an action of $\uqg$ for all $q>0$ such that $V \cong V^\bbC$ as $U_{q_0}(\fg)$-modules.
We also want these actions to be continuous in $q$ as in Lemma \ref{lem:continuity-of-uqg-actions}.
To do this, we choose a decomposition 
\[ V \cong \bigoplus_{j} V_{\lambda_j}   \]
of $V$ into simple $U_{q_0}(\fg)$-modules, and define
\[ V^\bbC =  \bigoplus_{j} V_{\lambda_j}^\bbC;   \]
it is clear that $V^\bbC$ has the desired properties. 
To extend this to tensor products, given finite-dimensional $U_{q_0}(\fg)$-modules $V_1, \dots, V_n$, we define
\[ (V_1 \otimes \dots \otimes V_n)^\bbC = V_1^\bbC \otimes \dots \otimes V_n^\bbC.  \]
Since all $\uqg$ act on the individual modules $V_i$, they also act on the tensor product, and again the actions are continuous in $q$.

\begin{rem}
  \label{rem:decomp-not-canonical}
  For a $U_{q_0}(\fg)$-module $X$, there is some ambiguity in the notation $X^\bbC$.
  We emphasize that whenever we present $X$ as a tensor product $X = V_1 \otimes \dots \otimes V_n$, by $X^\bbC$ we mean that we choose decompositions of the individual tensor factors $V_i$ and then take the tensor product of the individual $V_i^\bbC$ rather than an arbitrary decomposition of $X$ into simple modules.

  Of course, $X^\bbC$ also depends on the decomposition of each $V_i$ into simple submodules.
  This is not a problem, however, since we only perform this construction for one module at a time.
  In other words, we neither assert, nor require, that $X \mapsto X^\bbC$ is a functor.
\end{rem}

\begin{conv}
  \label{conv:dropping-upper-C}
  For a $U_{q_0}(\fg)$-module $V$, the vector space $V^\bbC$ carries an action of $\uqg$ for all $q>0$, and $V^\bbC \cong V$ as $U_{q_0}(\fg)$-modules.
  From now on, we will tacitly replace $V$ by $V^\bbC$, thereby allowing all $\uqg$ to act on $V$ itself.
  Remark \ref{rem:decomp-not-canonical} always applies to tensor products.

  For any $V$ and $W$, the $R$-matrices for the various $\uqg$ form a family of operators $R_q$ on $V \otimes W$, and we define a family $\bar{R}_q$ similarly.
  For $n \ge 2$, the cactus group $J_n$ acts on $V^{\otimes n}$ as in \S \ref{sec:cbdry_intro}, and we denote the resulting homomorphism by $\rho_q : J_n \to GL(V^{\otimes n})$.
\end{conv}

In order to discuss specializability and continuity of the coboundary operators, we need the following auxiliary result:

\begin{lem}
  \label{lem:matrix-coeffs-are-nice}
  Let $\lambda \in \cP^+$ and $B_\lambda$ the global crystal basis of $V_\lambda$.
  Let $V,W \in \ointvf$ and let $B$ be a basis for $V \otimes W$ as constructed above.
  \begin{enumerate}[(a)]
  \item The divided powers $E_\beta^{(N)}$, $F_\beta^{(N)}$ of the quantum root vectors act on $B_\lambda$ by Laurent polynomials.
  \item The $R$-matrix $R$ of $\uvq$ acts on $B$ by Laurent polynomials.
  \item Let $\bar{R} : V \otimes W \to V \otimes W$ be the action of $\bar{R}$ as in \S \ref{sec:cactus}.  The matrix coefficients of $\bar{R}$ with respect to the basis $B$ lie in $\qvm$, where $m$ is the smallest positive integer such that $m (\cP | \cP) \subseteq 2 \bbZ$.
  \end{enumerate}
\end{lem}

\begin{proof}
  (a) We already know that the divided powers of the generators $E_i$ and $F_i$ act on $B_\lambda$ by Laurent polynomials.
  Now we want to extend this to the quantum root vectors.

  Since $E_\beta = T(E_i)$ for some $i$, where $T$ is an automorphism coming from some element of the braid group of $\fg$, and since $\nu_\beta = \nu_i$, we have $E_\beta^{(N)} = T(E_i^{(N)})$, and similarly for $F_\beta$.

  The formulas in \cite{Lus10}*{41.1.2} can be adapted to show that $\uvz$ is preserved by the braid group action, and hence the divided powers of the quantum root vectors lie in $\uvz$. 
  (Lusztig works with a modified form $\dot{\mathbf{U}}$ of $\uvq$; note that the braid group action $\cT_i$ from \cite{KliSch97}*{6.2.2} coincides with Lusztig's $T''_{i,1}$.)
  
  Finally, elements of $\uvz$ act on $B_\lambda$ by Laurent polynomials, and in particular this is true for the divided powers of the quantum root vectors.

  (b) We can assume without loss of generality that $V$ and $W$ are simple.  We recall from \cite{KliSch97}*{8.3.3} that the action of the $R$-matrix in $V \otimes W$ is given by $R = D \circ \fR,$ where
  \begin{equation*}
    D(v \otimes w) = \nu^{(\wt(v) \mid \wt(w))} v \otimes w
  \end{equation*}
  for weight vectors $v,w$ (Klimyk-Schm\"udgen use $B_{VW}$ instead of $D$), and
  \begin{equation}
    \label{eq:rmatrix2}
    \fR = \sum_{t_1, \dots t_n = 0}^\infty \prod_{j=1}^n \frac{(1-\nu_{\beta_j}^{-2})^{t_j}}{[t_j]_{\nu_{\beta_j}}!} \nu_{\beta_j}^{t_j(t_j + 1)/2} E_{\beta_j}^{t_j} \otimes F_{\beta_j}^{t_j}.
  \end{equation}
  Here $n$ is the number of positive roots of $\fg$, the $\beta_j$ are the positive roots, and $\nu_{\beta} = \nu^{(\beta \mid \beta)}$. 
  The order in the product is determined by this decomposition as well. 
  Note that this is a finite sum since the root vectors act nilpotently in $V$ and $W$.
  
  We now show that $\mathfrak{R}$ acts on $B$ by Laurent polynomials. 
  Rearranging (\ref{eq:rmatrix2}) by dividing the $E_{\beta_j}^{t_j}$ term by $[t_j]_{\nu_{\beta_j}}!$, we get a linear combination of terms of the form $\displaystyle \prod_j E_{\beta_j}^{(t_j)} \otimes F_{\beta_j}^{t_j}$ with coefficients in $\zv$.
  The conclusion follows from applying part $\mathrm{(a)}$ to these terms.
  
  (c) By definition we have
  \[  \bar{R} = R(R_{21}R)^{-\frac12}.    \]
  By part $\mathrm{(b)}$ we only need to examine the action of the $R_{21}R$ factor.
  According to the construction of the basis $B$, it is enough to prove the statement when $V = V_\mu$ and $W = V_\nu$. 
  By \cite{KliSch97}*{8.4.2 Proposition 22}, $R_{21}R$ acts as the scalar
  \begin{equation}
    \label{eq:r21r-formula}
    \nu^{-(\mu \mid \mu + 2 \rho) - (\nu \mid \nu + 2 \rho) + (\lambda \mid \lambda + 2 \rho)}   
  \end{equation}
  on the $V_\lambda$-isotypic component of $V_\mu \otimes V_\nu$. Since the inverse square root of \eqref{eq:r21r-formula} belongs to $\qvm$, the result follows. 
\end{proof}

\begin{prop}
  \label{prop:switching-is-continuous}
  Let $V,W \in \ointqf$ and $n \ge 2$.
  Allow all $\uqg$ to act on $V$ and $W$ and let $R_q, \bar{R}_q$, and $\rho_q : J_n \to GL(V^{\otimes n})$ be as in Convention \ref{conv:dropping-upper-C}.
  \begin{enumerate}[(a)]
  \item The family $R_q$ is continuous for $q \neq 1$ and extends continuously to $q=1$, with $R_1 = \id$.
  \item The family $\bar{R}_q$ is continuous for $q \neq 1$ and extends continuously to $q=1$, with $\bar{R}_1 = \id$.
  \item For any $x \in J_n$ the family $\rho_q(x) \in GL(V^{\otimes n})$ is continuous for $q \neq 1$ and extends continuously to $q=1$.  
    The resulting homomorphism $\rho_1 : J_n \to GL(V^{\otimes n})$ factors through the canonical action of $S_n$ by means of the homomorphism (\ref{eq:cactus_quotient}).
  \end{enumerate}
\end{prop}

\begin{proof}
  (a) By part $\mathrm{(b)}$ of Lemma \ref{lem:matrix-coeffs-are-nice}, and since Laurent polynomials in $q$ are continuous for $q \in (0,\infty)$, we conclude that $q \mapsto R_q $ is continuous. At $q=1$, the terms $(1-q_{\beta_j}^{-2})^{t_j}$ of \eqref{eq:rmatrix2} vanish except when $t_j = 0$, leaving only the term $t_1 = \dots = t_n = 0$, which acts as the identity. Hence $R_1 = \id$.
  
  (b) Recall that $\bar{R}_q = R_q(R_{21}R)_q^{-\frac{1}{2}}$; $(R_{21})_q = \tau \circ R_q \circ \tau$ is continuous in $q$, so we need only the fact that $T \mapsto T^{-\frac{1}{2}}$ is a continuous function on operators with positive spectrum. 
  This can be seen using the holomorphic functional calculus for the Banach algebra $\End(V \otimes W)$ and the fact that $x \mapsto x^{-\frac{1}{2}}$ is holomorphic on the right half-plane. Since $R_q \to \id$ by $\mathrm{(a)}$, the conclusion follows.
  
  (c) Apply part $\mathrm{(b)}$ to the operators $\sigma_{p,r,t}$ defined in \eqref{eq:sigma_prq}, which generate $J_n$. 
\end{proof}

The following result is the cornerstone of the proof that $S_q(V)$ and $\ext_q(V_\o)$ are commutative algebras in $\ointq$ and $\cS \ointq$, respectively.

\begin{prop}
  \label{prop:cactus-group-fixes-symq}
  The space $S_q^n V$ is fixed pointwise by the action of the cactus group $J_n$.
  Similarly, the space $\ext_q^n V_\o$ is fixed pointwise by the action of the cactus group $J_n$ coming from the coboundary structure on $\cS\ointq$ as in \S \ref{sec:super-reps}.
\end{prop}

\begin{proof}
  We focus on $S_q^nV$, but we will indicate briefly at the end how to modify the proof to handle the antisymmetric case.

  By induction on $n$, we can assume that $S^n_q V$ is fixed pointwise by all $\rho_q(s_{p,t})$ with $t-p < n-1$.
  Indeed, there are two natural embeddings of $J_{n-1}$ into $J_n$, as elements acting on either the first or last $n-1$ tensor factors of $V^{\otimes n}$.
  The condition $t-p < n-1$ means that $s_{p,t}$ acts on at most $n-1$ consecutive tensor factors, and hence it can be regarded as an element of $J_{n-1}$ via one of these two embeddings.

  Since $s_{1,n}$ together with these $s_{p,t}$ generate all of $J_n$ (see \cite{HenKam06}*{Lemma 3 (v)}), it is enough to show that $\rho_q(s_{1,n})$ acts as the identity on $S^n_q V$.
  By \cite{HenKam06}*{Lemma 4}, $s_{1,n}$ normalizes the subgroup of $J_n$ generated by the $s_{i,i+1}$ and so the submodule $S^n_q V$ is invariant under $\rho_q(s_{1,n})$. 
  Since $s_{1,n}$ is involutive, its eigenvalues can be only $\pm 1$, so we must  show that  $-1$ does not occur as an eigenvalue on $S^n_q V$.

  Suppose that $\rho_q(s_{1,n})$ has an eigenvalue $-1$ on  $S^n_q V$. Our plan is to show that $\rho_q(s_{1,n})$ must therefore have $-1$ as an eigenvalue on $S^n_q V$ for \emph{all} transcendental $q$. 
  We will then employ a limiting argument to obtain a contradiction with the continuity in $q$ of the actions of $J_n$.

  By part $\mathrm{(c)}$ of Lemma \ref{lem:matrix-coeffs-are-nice} (which carries over to the $q \in \bbC$ setting because $q$ is transcendental), there is an eigenvector $v_q \in S^n_qV$ for $\rho_q(s_{1,n})$ with eigenvalue $-1$ which is in the $\Qq$-span of the basis $B$ for $V^{\otimes n}$. 
  Then the equality $\rho_q(s_{1,n})v_q = -v_q$ becomes a system of rational equations in $q$. 
  Since $q$ is transcendental, this system of equations must hold identically wherever they are defined, and in particular they hold if we replace $q$ with any positive number $\tilde{q}$. 
  Since the matrix coefficients of $\rho_q(s_{1,n})$ and $\rho_{\tilde{q}}(s_{1,n})$ are obtained from those of Lemma \ref{lem:matrix-coeffs-are-nice} $\mathrm{(c)}$ by replacing $\nu$ with $q$ and $\tilde{q}$, respectively, this implies that $\rho_{\tilde{q}}(s_{1,n}) v_{\tilde{q}} = - v_{\tilde{q}}$, where by $v_{\tilde{q}}$ we mean replacing $q$ with $\tilde{q}$ in the coordinates of $v_q$.

  In order to pass to the limit at $\tilde{q}=1$, we scale the $v_{\tilde{q}}$ so that they are uniformly bounded and uniformly bounded away from zero with respect to an arbitrary but fixed inner product.
  By compactness, there is a nonzero limit point $v_1$ of the $v_{\tilde{q}}$ as $\tilde{q} \to 1$, which will be an eigenvector for $\rho_1(s_{1,n})$ with eigenvalue $-1$ by Proposition \ref{prop:switching-is-continuous} $\mathrm{(c)}$.
  Moreover $v_1 \in S^n V$ also by Proposition \ref{prop:switching-is-continuous} $\mathrm{(c)}$.
  This is a contradiction: $\rho_1(s_{1,n})$ acts trivially on $S^n V$ since $\rho_1$ factors through the symmetric group action.

  The proof for $\ext_q^n V_\o$ is the same, recalling that the minus signs from the definition of $\ext_q^n V$ in \eqref{eq:symvectors} are absorbed into the coboundary structure of $\cS \ointq$.
\end{proof}

\begin{rem}
  \label{rem:cactus-group-fixes-symq}
  Note that the elements $s_{i,i+1}$ do not generate the cactus group $J_n$.
  Nevertheless, Proposition \ref{prop:cactus-group-fixes-symq} says that if an element of $V^{\otimes n}$ is fixed by all $s_{i,i+1}$ then it is fixed by all of $J_n$.
  There is no reason to expect that an arbitrary coboundary structure has this property.
\end{rem}


\subsection{Proof of the theorem}

We first prove commutativity:

\begin{prop}
  \label{prop:commutative-algebra}
  For a transcendental number $q > 0$ and a module $V$ in $\ointqf$, the algebra $S_q(V)$ is commutative in $\ointq$, while $\ext_q(V_\o)$ is commutative in $\cS\ointq$. 
\end{prop}

\begin{proof}
  We will only prove the statement concerning $S_q(V)$.
  Only minimal changes are required in order to adapt the proof to $\ext_q(V_\o)$.
  
  We show that $S_q(V)$ is commutative by lifting everything to the tensor algebra. 
  Let $n \ge 2$ and $1 \le r < n$. 
  We want to show that the lower triangle in the diagram  
  \[
  \begin{tikzpicture}[anchor=base,cross line/.style={preaction={draw=white,-,line width=6pt}}]
    \path (0,0) node (1) {$\scriptstyle T^r\otimes T^{n-r}$} +(6,0) node (2) {$\scriptstyle T^n$} +(2,-1) node (3) {$\scriptstyle T^{n-r}\otimes T^r$} +(0,-4) node (4) {$\scriptstyle S^r\otimes S^{n-r}$} +(6,-4) node (5) {$\scriptstyle S^n$} +(2,-5) node (6) {$\scriptstyle S^{n-r}\otimes S^r$}; 
    \draw[->] (1) -- (2) node[pos=.5,auto] {$\scriptstyle m$};
    \draw[->] (1) -- (3) node[pos=.5,auto,swap] {$\scriptstyle \gamma_q$};
    \draw[->] (3) -- (2) node[pos=.5,auto,swap] {$\scriptstyle m$};
    \draw[->] (4) -- (5) node[pos=.5,auto] {$\scriptstyle m$};
    \draw[->] (4) -- (6) node[pos=.5,auto,swap] {$\scriptstyle \gamma_q$};
    \draw[->] (6) -- (5) node[pos=.5,auto,swap] {$\scriptstyle m$};
    \draw[->] (1) -- (4) node[pos=.5,auto,swap] {$\scriptstyle \pi\otimes\pi$};
    \draw[->] (2) -- (5) node[pos=.5,auto] {$\scriptstyle \pi$};
    \draw[->,cross line] (3) -- (6) node[pos=.5,auto] {$\scriptstyle \pi\otimes\pi$};
  \end{tikzpicture}
  \]
  commutes.
  Here $S^n = S^n_q(V)$ and $T^n = V^{\otimes n}$, while $\pi : T^n \to S^n$ is the quotient map,  $\gamma_q = \rho_q(\sigma_{1,r,n})$ (see~\S \ref{sec:cbdry_intro} and Convention \ref{conv:dropping-upper-C}), and $m$ denotes multiplication.

  Since all of the squares in the diagram commute and $\pi \otimes \pi$ is surjective, a diagram chase shows that it suffices to prove that 
  \begin{equation*}
    \pi \circ m \circ \gamma_q = \pi \circ m : T^r \otimes T^{n-r} \to S^n.
  \end{equation*}
  The two upper $m$ arrows are the canonical identifications of $T^r \otimes T^{n-r}$ and $T^{n-r} \otimes T^r$ with $T^n$, so we can regard $\gamma_q$ as an operator on $T^n$.
  With this identification in mind, we need to prove that $\pi \circ \gamma_q = \pi$, or in other words that $\ran (\gamma_q - \id) \subseteq \ker (\pi)$.
  Equivalently, we will show that $\ker(\pi)^\perp \subseteq \ran (\gamma_q - \id)^\perp$.
  Equation \eqref{eq:emb1} says that $\ker(\pi)^\perp = S^n_q V$ (note that $\rho_q(s_{i,i+1}) = \sigma_i$ in the notation of \eqref{eq:emb1}).
  On the other hand, since $\gamma_q$ is unitary, we have
  \[ \ran(\gamma_q - \id)^\perp = \ker(\gamma_q^* - \id) = \ker(\gamma_q^{-1} - \id),    \]
  so we need to show that $\gamma_q$ acts as the identity on $S^n_q V$.
  This follows from Proposition \ref{prop:cactus-group-fixes-symq}.

  To prove that $\ext_q(V_\o)$ is commutative in $\cS\ointq$, merely replace $\gamma$ with its super-analogue as in \S \ref{sec:super-reps} and $S^n_q(V)$ with $\ext^n_q(V_\o)$.
\end{proof}

Finally, everything we need for the proof of the main result of this section is in place.

\begin{proof}[Proof of Theorem \ref{thm:enveloping_comm_alg}]
  As above, we will only deal with the statement concerning $S_q(V)$, the case of $\ext_q(V_\o)$ being analogous.

  We know from Proposition \ref{prop:commutative-algebra} that $S_q(V)$ is commutative, so it remains only to verify the universality property. 
  Suppose that $A$ is a commutative algebra in $\ointq$ and $f : V \to A$ is a module map. By the universal property of the tensor algebra, $f$ lifts uniquely to a morphism of algebras $\tilde{f} : T(V) \to A$. 
  We only need to show that $\tilde{f}$ factors as 
  \[   T(V) \overset{\pi}{\longrightarrow} S_q(V)  \overset{\hat{f}}{\longrightarrow}  A,   \]
  or in other words that $\langle \Lambda_q^2 V \rangle \subseteq \ker(\tilde{f})$. 
  The uniqueness of $\hat{f}$ will then follow from uniqueness of $\tilde{f}$ and surjectivity of the quotient map $T(V) \to S_q(V)$.

  We now show that $\Lambda^2_q V \subseteq \ker(\tilde{f})$. 
  On the degree two component of $T(V)$, $\tilde{f}$ is defined by $\tilde{f}(v \otimes w) = m(f(v) \otimes f(w))$, where $m$ is the multiplication map of $A$. 
  By the functoriality of the coboundary operators, $(f \otimes f) (\Lambda^2_qV )\subseteq \Lambda^2_q A$. 
  Hence it suffices to show that $m$ vanishes on $\Lambda^2_q A$. 
  But this is immediate from the definition of commutativity.
\end{proof}


\section{The Grothendieck group of $\uqg$-reps}
\label{sec:groth}

For $q > 0$, let $\ointqf$ denote the full subcategory of $\ointq$ consisting of finite-dimensional modules.
The categories $\ointqf$ are all semisimple (every module is a finite direct sum of simple modules), their simple objects are all indexed by $\cP^+$, and their fusion rules are the same.
Hence, this identification of simple objects induces a canonical isomorphism of the Grothendieck semirings of these categories; we denote this common semiring by $K^+$, and the corresponding Grothendieck ring by $K$.
We denote multiplication in $K$ by $\cdot$ or simply by juxtaposition when appropriate.
Defining $K^+$ to be the positive cone determines a partial order $\le$ on $K$.

We will often abuse notation by denoting a $\uqg$-representation and the corresponding element of $K$ by the same symbol. 
For $\lambda\in\cP^+$, the notion of multiplicity of $V_\lambda$ in a $\uqg$-module extends in the evident manner to a notion of multiplicity of $V_\lambda$ in an element of $K^+$ or $K$.
Note that the multiplicity of $V_\lambda$ in an element of $K$ can be negative.

For any non-negative integer $n$, the endofunctors $V \mapsto S_q^n V$ and $V \mapsto \ext_q^n V$ of the categories $\ointqf$ descend to functions $S_q^n, \ext_q^n : K^+ \to K^+$.  
Note that $S_1^n$ and $\ext_1^n$ are induced by the usual symmetric and exterior powers of $\fg$-modules, and we denote these just by $S^n$ and $\ext^n$, respectively.

Our goal in this section is to prove the following:

\begin{thm}
  \label{thm:symext}
  The identity
  \begin{equation}
    S_q^3 V -\ext_q^3 V = S^3 V - \ext^3 V
  \end{equation}	
  holds in $K$ for any $V\in K^+$ and any transcendental $q>0$.
\end{thm}

As with Theorem \ref{thm:enveloping_comm_alg}, we expect this result to hold for all positive $q$.

\begin{rem}
  \label{rem:grothendieck-semigroups}
  As an abelian semigroup, $K^+$ is freely generated by (the classes of) the modules $V_\lambda$ for $\lambda \in \cP^+$.  
  It follows that $K^+$ injects into $K$.
  Hence the statement of Theorem \ref{thm:symext} can be lifted to the equality
  \[      S_q^3 V + \ext^3 V = S^3 V + \ext_q^3 V   \] 
  in $K^+$.
\end{rem}

The motivation for Theorem \ref{thm:symext} comes from \cite{BerZwi08}*{Conjecture 2.26}, which asserts a ``numerical Koszul duality'' between quantum symmetric and exterior algebras.
Recall that the \emph{Hilbert series} $h(A,t)$ of a locally finite-dimensional graded algebra $A$ is the generating function for the dimensions of the graded components of $A$ \cite{PolPos05}*{\S 2.2}.
If $A$ is a Koszul algebra with quadratic dual algebra $A^!$, then we have
\begin{equation}
  \label{eq:hilbert-series}
  h(A,t) \cdot h(A^!,-t) = 1;
\end{equation}
see \cite{PolPos05}*{Corollary 2.2}.
The quadratic dual of $S_q(V_\lambda)$ is $\ext_q(V_\lambda^*)$ \cite{BerZwi08}*{Proposition 2.11}.
These algebras are Koszul if $V_\lambda$ is flat \cite{BerZwi08}*{Proposition 2.28}, but there is no reason to expect that \eqref{eq:hilbert-series} holds in general for $A = S_q(V_\lambda)$.
However, the conjecture states that it does hold up to third order, i.e.
\[ h(S_q(V_\lambda),t) \cdot h(\ext_q(V_\lambda^*), -t) = 1 + O(t^4). \]
This is equivalent to the equality
\begin{equation}
  \label{eq:numerical-Koszul-duality-equiv}
  \dim S_q^3 V_\lambda - \dim\ext_q^3 V_\lambda^* = (\dim V_\lambda)^2.
\end{equation}
Since $\ext_q(V_\lambda)$ has the same Hilbert series as $\ext_q(V_\lambda^*)$ \cite{BerZwi08}*{Proposition 2.11}, Theorem \ref{thm:symext} confirms \eqref{eq:numerical-Koszul-duality-equiv}, and hence also the conjecture.
(The parameter $q$ is taken to be a formal variable in \cite{BerZwi08}, which is equivalent to taking $q \in \bbC^\times$ transcendental.)
In fact, it proves a slightly stronger version, accounting for \emph{all} finite-dimensional modules, as opposed to just simple ones.
This result means that the quantum symmetric and exterior algebras exhibit the same amount of collapsing in their degree three components.

Additionally, Theorem \ref{thm:symext} provides positive evidence for \cite{Zwi}*{Conjecture 7.3}. 
For a module $V \in \ointqf$, \cite{Zwi}*{Definition 6.18} constructs elements $S_\l^3 V$ and $\ext_\l^3 V$ of $K^+$ by removing from $S^2 V \otimes V$ and $\ext^2 V \otimes V$, respectively, their greatest common submodule.
This description is equivalent to the definition of $S^3_\l V$ and $\ext^3_\l V$ given in \cite{Zwi}.
More precisely, if we denote by $W$ the infimum of $S^2V \cdot V$ and $\ext^2 V \cdot V$ in the lattice $(K, \le)$, then
\[   S_\l^3 V  = S^2 V \cdot V - W, \quad \ext_\l^3 V = \ext^2V \cdot V - W.  \]

The proof of \cite{BerZwi08}*{Lemma 2.30} shows that the inequalities $S_\l^3 V \le S_q^3 V$ and $\ext_\l^3 V \le \ext_q^3 V$ hold in $K$ (their result is stated for simple $V$, but the proof works in general). 
Finally, \cite{Zwi}*{Conjecture 7.3} states, essentially, that these inequalities are in fact equalities when $V$ is simple and $q$ is transcendental.
This means that $S^3_q V$ and $\ext^3_q V$ display the maximal amount of collapsing when $q$ is transcendental.
By construction, $S_\l^3 V - \ext_\l^3 V = (S^2 V - \ext^2 V)V$, and we will see below that $(S^2 V - \ext^2 V)V = S^3 V - \ext^3 V $, so Theorem \ref{thm:symext} is positive evidence for the conjecture.   

For an arbitrary element $V\in K^+$, we also denote by $V$ a lift of this element to a $\uqg$-module for some $q>0$. 
Moreover, as in Convention \ref{conv:dropping-upper-C}, we will assume that all algebras $\uqg$ and all of the various braid or coboundary operators for $q>0$ act simultaneously on $V$ and its tensor powers. 

The strategy for proving Theorem \ref{thm:symext} consists of showing that when $q>0$ is either transcendental or $1$ the identity \[ S_q^3V-\ext_q^3V=(S_q^2V-\ext_q^2 V)V \] holds in $K$, and also that the right hand expression does not depend on $q$. 
In order to prove the latter statement, we require the following result:

\begin{lem}
  \label{lem:cts-map-to-grassmannian}
  Let $V_1, \dots, V_n \in K^+$ and $\lambda \in \cP^+$.  
  Regard $V = V_1 \otimes \dots \otimes V_n$ as a complex vector space with actions of $\uqg$ for all $q > 0$ as in Convention \ref{conv:dropping-upper-C}.
  Denote by $V_q^\lambda \subseteq V$ the space of highest weight vectors of weight $\lambda$ for the action of $\uqg$.
  Then $q \mapsto V_q^\lambda$ is a continuous map into the Grassmann manifold $\Gr(V)$ of $V$.
\end{lem}

\begin{proof}
  Since tensor products decompose with the same weight multiplicities as they do classically, the dimension of $V_q^\lambda$ is independent of $q$.
  Thus the range of $q \mapsto V_q^\lambda$ is contained in a single component of the Grassmannian.

  Let $(q_n)$ be a sequence of positive real numbers converging to some $q_0 > 0$.
  We want to show that $(V_{q_n}^\lambda)$ converges to $V_{q_0}^\lambda$ in $\Gr(V)$.
  If not, there is an open neighborhood $\cU$ of $V_{q_0}^\lambda$ such that some subsequence of $(V_{q_n}^\lambda)$ avoids $\cU$.
  Passing to this subsequence, we may assume that $V_{q_n}^\lambda \notin \cU$ for all $n$.

  By compactness of $\Gr(V)$, a subsequence of $V_{q_n}^\lambda$ converges to some point $W \notin \cU$.
  It follows from Lemma \ref{lem:continuity-of-uqg-actions} that $W$ consists of highest weight vectors of weight $\lambda$ for the action of $U_{q_0}(\fg)$.
  Hence $W \subseteq V_{q_0}^\lambda$ by definition, so they must be equal since their dimensions coincide.
  But this contradicts $W \notin \cU$.
\end{proof}

\begin{lem}
  \label{lem:squares_are_classical}
  For any $V\in K^+$, the elements $S_q^2V$ and $\ext_q^2V$ of $K^+$ do not depend on $q>0$.  	
\end{lem}

\begin{proof}
  As in Convention \ref{conv:dropping-upper-C}, we consider $V$ to be a complex vector spaces on which all the algebras $\uqg$ act.
  For fixed $\lambda \in \cP^+$, we will show that the multiplicity of $V_\lambda$ in $S^2_q V$ is independent of $q$.

  For each $q>0$, let $\sigma_q = \rho_q(s_{1,2})$ be the coboundary operator for the action of $\uqg$.
  Fix an arbitrary inner product $(\cdot,\cdot)$ on $V\otimes V$, and for each $q$ let $P_q^\lambda$ be the orthogonal projection with respect to $(\cdot,\cdot)$ on the space $(V \otimes V)_q^\lambda$ of highest weight vectors of weight $\lambda$ in $V \otimes V$.
  By Lemma \ref{lem:cts-map-to-grassmannian} applied to $V \otimes V$, $q \mapsto (V \otimes V)_q^\lambda$ is a continuous map to $\Gr(V \otimes V)$, and hence $q \mapsto P_q^\lambda$ is a continuous family of operators.
  
  Note that $\frac{1 + \sigma_q}{2}$ is the projection of $V \otimes V$ onto $S^2_q V$.
  The coboundary operator $\sigma_q$ preserves $(V \otimes V)^\lambda_q$ since it is a $\uqg$-module map, so $\sigma_q$ commutes with $P_q^\lambda$, and hence $\frac{1+\sigma_q}2  P_q^\lambda$ is an idempotent which projects onto the space of highest weight vectors of weight $\lambda$ in $S_q^2 V$.
  The multiplicity of $V_\lambda$ in $S_q^2V$ with respect to the $\uqg$-action is, then, the trace of $\frac{1+\sigma_q}2  P_q^\lambda$.

  Since the trace is constant on a continuously varying family of idempotents, this means that the multiplicity of $V_\lambda$ in $S^2_q V$ is independent of $q$.
  Since this is true for any $\lambda$, the result follows.
  The same argument works for $\ext_q^2 V$ after replacing $\frac{1+\sigma_q}2$ by $\frac{1-\sigma_q}2$.    
\end{proof}

\begin{rem}
  \label{rem:collapse}
  It is a consequence of \cite{BerZwi08}*{Theorem 2.21} that for transcendental $q$, the inequalities $S_q^nV \le S^nV$ and $\ext_q^nV \le \ext^nV$ hold for any $V \in K^+$. 
  In particular, when $n=2$, a dimension count shows that these are actually equalities. 
  This provides a simple proof for Lemma \ref{lem:squares_are_classical} for the case when $q$ is either transcendental or $1$.
\end{rem}

We now make some elementary observations, needed below, on the 3-fruit cactus group $J_3$. 
It is a simple consequence of the definition of $J_n$ from \S \ref{sec:cbdry_intro} (using the notation introduced there, and suppressing commas between indices) that $J_3$ is generated by the three elements $a=s_{12}$, $b=s_{23}$, and $\psi = s_{13}s_{12}$; of course, we could have used $s_{13}$ as the third element, but working with $\psi$ will be more convenient. 
For $q>0$ and $V\in\ointq$, the action of $J_3$ on $V^{\otimes 3}$ coming from the usual coboundary structure on $\ointq$ (see \ref{sec:cactus}) is 
\[  a \mapsto \sigma_{V,V} \otimes \id_V, \quad b \mapsto \id_V \otimes \sigma_{V,V}, \quad \psi \mapsto \sigma_{V\otimes V,V}.
\]

\begin{lem}
  \label{lem:rels_for_J3}
  The following relations hold in the cactus group $J_3$:
  \begin{enumerate}[(a)]
  \item $a^2 = b^2 = 1$;
  \item $\psi a=b\psi$;
  \item $a\psi a=\psi^{-1}$.
  \end{enumerate}
\end{lem}

\begin{rem}
  \label{rem:generators-of-J3}
  It follows immediately from part (b) that $J_3$ is generated by $a$ and $\psi$.
\end{rem}

\begin{proof}
  (a) follows from involutivity of the generators $s_{p,t}$ of the cactus group.
  Using $\psi a = s_{13}$, (b) reads $s_{13}=s_{23}s_{13}s_{12}$; this follows from the relation $s_{13}s_{12} = s_{23}s_{13}$.
  Since $\psi a=s_{13}$ is involutive, it must be equal to its inverse, $a\psi^{-1}$. 
  Multiplying by $a$ on the left and using $a^2=1$, we get (c).
\end{proof}

We will be working with the families $a_q=\rho_q(a)$, $b_q=\rho_q(b)$ and $\psi_q=\rho_q(\psi)$ of operators in $GL(V^{\otimes 3})$, indexed by $q>0$, where $\rho_q$ is as defined in Convention \ref{conv:dropping-upper-C}.

\begin{lem}
  \label{lem:alternative_descr_symext}
  Let $q>0$ be either transcendental or $1$, and $V\in\ointqf$. 
  Then the quantum symmetric and exterior cubes of $V$ can be described as follows:
  \begin{equation}
    \label{eq:sym3}
    S_q^3V=\left\{v\in V^{\otimes 3} \mid a_qv=v,\ \psi_q v=v\right\};
  \end{equation}
  \begin{equation}
    \label{eq:ext3}
    \ext_q^3V = \left \{v\in V^{\otimes 3} \mid a_qv=-v,\ \psi_q v=v \right\}.
  \end{equation}
\end{lem}
\begin{proof}
  As part of the proof of Proposition \ref{prop:commutative-algebra}, we showed that $S_q^nV$ is fixed by the entire cactus group $J_n$ when $q>0$ is transcendental.
  This holds trivially when $q=1$. 
  Hence, $S_q^3V$ is certainly contained in the right hand side of \eqref{eq:sym3}. 
  The opposite inclusion follows from the fact that $a$ and $\psi$ generate $J_3$ (see Remark \ref{rem:generators-of-J3}), so anything fixed by $a_q$ and $\psi_q$ will also be fixed by $b_q$.  
  
  The argument for $\ext_q^3$ is analogous. 
\end{proof}

\begin{proof}[Proof of Theorem \ref{thm:symext}]
  As explained above, Lemma \ref{lem:squares_are_classical} has reduced the problem to proving that the identity
  \begin{equation}
    \label{eq:proof_of_thm:symext}
    S_q^3V-\ext_q^3V=(S_q^2V-\ext_q^2V)V
  \end{equation}
  holds in $K$ when $q>0$ is either transcendental or $1$. Let us rephrase this slightly. 

  For $\lambda\in\cP^+$, denote by $a_q^\lambda$, $b_q^\lambda$ and $\psi_q^\lambda$ the restrictions of $a_q$, $b_q$ and $\psi_q$ respectively to the space $(V\otimes V\otimes V)^\lambda$ of highest weight vectors in $V\otimes V\otimes V$ of weight $\lambda$. 

  Lemma \ref{lem:alternative_descr_symext} implies that the multiplicity of $V_\lambda$ in the left hand side of \eqref{eq:proof_of_thm:symext} is
  \begin{equation}
    \label{eq:left_lambda}
    \dim(\ker(a_q^\lambda-1)\cap\ker(\psi_q^\lambda-1))-\dim(\ker(a_q^\lambda+1)\cap\ker(\psi_q^\lambda-1)).
  \end{equation}  
  On the other hand, since $S_q^2V$ and $\ext_q^2V$ are precisely the $1$ and $(-1)$-eigenspaces of $a_q$ respectively, the multiplicity of $V_\lambda$ in the right hand side of \eqref{eq:proof_of_thm:symext} is  
  \begin{equation}
    \label{eq:right_lambda}
    \dim\ker(a_q^\lambda-1)-\dim\ker(a_q^\lambda+1).
  \end{equation}  
  What we have to prove, then, is that the difference in dimension between the $1$ and $(-1)$-eigenspaces of $a_q^\lambda$ does not change when we restrict to the $a_q^\lambda$-invariant subspace $\ker(\psi_q-1)$ of $(V\otimes V\otimes V)^\lambda$. Note that as a consequence of part (c) of Lemma \ref{lem:rels_for_J3}, $a_q^\lambda$ does indeed act on $\ker(\psi_q^\lambda-1)$.

  Let $t\in\bbC^\times$ be an arbitrary non-zero complex number. 
  Since $a_q^\lambda$ is an involution, the same part (c) of Lemma \ref{lem:rels_for_J3} implies that $a_q^\lambda$ implements an isomorphism of the $t$-eigenspace of $\psi_q^\lambda$ onto the $t^{-1}$-eigenspace, and vice versa. 
  It follows that the restrictions of $\frac{1\pm a_q^\lambda}2$ to $\bigoplus_{t\ne t^{-1}}\ker(\psi_q^\lambda-t)$ are projections of equal ranks, and hence the difference between \eqref{eq:right_lambda} and \eqref{eq:left_lambda} is 
  \begin{equation}
    \label{eq:right-left_lambda}
    \dim(\ker(a_q^\lambda-1)\cap\ker(\psi_q^\lambda+1))-\dim(\ker(a_q^\lambda+1)\cap\ker(\psi_q^\lambda+1)).
  \end{equation}
  In other words, the only discrepancy between the two arises from the action of $a_q^\lambda$ on the $(-1)$-eigenspace of $\psi_q^\lambda$. 

  For $q=1$, $\psi_q$ is a cyclic permutation of order $3$ and hence cannot have $-1$ as an eigenvalue. 
  On the other hand, for transcendental $q>0$, an argument similar to the one used in the proof of Proposition \ref{prop:commutative-algebra} shows that if $\psi_q$ has $-1$ as an eigenvalue, then so does $\psi_1$; we have just argued that this is not the case. 
  In conclusion, \eqref{eq:right-left_lambda} is zero for the simple reason that $\ker(\psi_q^\lambda+1)$ is trivial, and hence \eqref{eq:left_lambda} $=$ \eqref{eq:right_lambda}, as desired.  
\end{proof}


\end{document}